\documentclass{amsart}
\usepackage{amsfonts,amsmath,amsthm,color, amssymb}
\usepackage[all]{xy}

\newtheorem{theorem}{Theorem}[section]
\newtheorem{lemma}[theorem]{Lemma}
\newtheorem{proposition}[theorem]{Proposition}
\newtheorem{corollary}[theorem]{Corollary}

\theoremstyle{definition}

\newtheorem{example}[theorem]{Example}

\theoremstyle{remark}
\newtheorem{remark}[theorem]{Remark}

\numberwithin{equation}{section}

\newcommand{\Set}{\mathbf{Set}}
\newcommand{\Art}{\mathbf{Art}}

\newcommand{\Oh}{\mathcal{O}}

\newcommand{\Def}{\operatorname{Def}}

\newcommand{\Spec}{\operatorname{Spec}}

\newcommand{\C}{\mathbb{C}}

\newcommand{\N}{\mathbb{N}}
\newcommand{\R}{\mathbb{R}}
\newcommand{\Z}{\mathbb{Z}}

\newcommand{\K}{\mathbb{K}\,}

\begin{document}

\title{Diffeomorphism Classes of Calabi-Yau varieties}

\author{Gilberto Bini}
\address{\newline
Universit\`a degli Studi di Milano,\hfill\newline
Dipartimento di Matematica \lq\lq F.Enriques\rq\rq,\hfill\newline
Via C. Saldini 50,
20133 Milano, Italy.}
\email{gilberto.bini@unimi.it }

\author{Donatella Iacono}
\address{\newline  Universit\`a degli Studi di Bari,
\newline Dipartimento di Matematica,
\hfill\newline Via E. Orabona 4,
70125 Bari, Italy.}
 \email{donatella.iacono@uniba.it }

\begin{abstract}
In this article we investigate diffeomorphism classes of Calabi-Yau threefolds. In particular, we focus on those embedded in toric Fano manifolds. Along the way, we give various examples and conclude with a curious remark regarding mirror symmetry.
\end{abstract}

\maketitle

\section{Introduction}

A longstanding problem in geometry is the  classification of geometric objects
up to isomorphism. For example, from a topological point of view, we are interested in classifying objects up to homeomorphism.
In  Differential Geometry, the classification is
up to diffeomorphism and in complex geometry,   we look for a classification up
to (analytic) isomorphism.

This is the starting point for the construction of the moduli space.  The main goal is the
classification of families of these geometric objects (up to equivalence) so that
the classifying space, the so called {\em moduli space}, is a reasonable geometric space. Roughly speaking,
the moduli space is a parameter space that classifies these objects, in the sense that its points
 parametrise the geometric objects that we are considering.
One of the easiest examples is the collection of all the lines (through the origin) in three dimensional
space. The space that classifies this collection is well known and has a nice geometric
structure: it is the projective plane (a smooth and compact manifold). As another example, we
can consider the space that classifies, up to isomorphism, smooth rational curves of genus
zero with 3 distinct marked points. It turns out that this space is just a point, since any triple
of distinct points on a projective line can be sent in a distinct triple by an automorphism.

Unfortunately, the  general  situation turns out to be very complicated.
In complex dimension one,  we would like to classify all smooth curves, i.e., Riemann surfaces up to isomorphism.
The classification can be carried out by using the genus $g$ of the curve. For $g \geq 1$ the moduli space $\mathcal{M}_g$
is well understood and has a rich geometric structure. We also observe that in this case all the objects are projective, i.e., all smooth curves embedded in projective space.

In dimension two, the classification of compact complex surfaces is more involved than that in dimension one.
It turns out that it is convenient  to  classify birational classes of surfaces. Then,  for
every birational class there is  a unique minimal  model, that has to be classified.

In dimension higher or equal than three, the classification is quite far from being complete. The  idea is to
generalize the technique used for dimension two and this has developed the so called
Minimal Model Program.  This classification is not concluded yet and
already in dimension three there are many technical issues that have to be understood such as the uniqueness of the minimal model.

Motivated by a better  understanding of this classification, we are interested in the role
played by Calabi-Yau manifolds. First of all, the classification of 3-dimensional algebraic varieties has still some gaps
due to the lack of understanding of Calabi-Yau threefolds. Moreover, the moduli space of Calabi-Yau varieties has received attention by theoretical physicists, since these geometric objects are important
for mirror symmetry, cohomological quantum field theory and
string theory, branches of physics dealing with general relativity and quantum mechanics.

In dimension one, Calabi-Yau curves are genus $1$ curves and they are all homeomorphic each other.
These are not isomorphic and they  are classified by the so called $j$-invariant.
In dimension two, Calabi-Yau surfaces are called K3 surfaces and they are all homeomorphic each other. Also in this case they are not all isomorphic; moreover, there exist  K3 surfaces that are not projective. We also remark that K3 surfaces are extensively studied and they  play a central role in algebraic and complex geometry.

In higher dimension, the classification of  Calabi-Yau manifolds is quite hard and many questions are still open also in the topological
setting. For example, the topology of Calabi-Yau manifolds is not uniquely determined for dimension greater or equal than three. It is also not
known if there   are only finitely many topological  types of Calabi-Yau threefolds.

From the differential point of view, C.T.C. Wall described the invariants that determine the diffeomorphism  type of
closed, smooth, simply connected 6-manifolds with torsion free cohomology \cite{wall}. In particular,
the Hodge data, the triple intersection in cohomology and the second Chern class completely determine the diffeomorphism type of a simply
connected Calabi-Yau threefold.
 Recently, A. Kanazawa, P. M. H. Wilson \cite{Kanazawa}, refined the theorem  by  Wall, providing some inequalities on the
 invariants, which hold in the case of Calabi-Yau threefolds.

The setting is very complicated. An interesting task is to find new examples of Calabi-Yau threefolds.
The best known example of Calabi-Yau threefolds is the smooth quintic hypersurface in projective space $\mathbb{P}^4$, which is defined by a homogeneous polynomial of degree five. Actually, this example can be generalized to construct the majority of all known Calabi-Yau varieties. Indeed, projective space $\mathbb{P}^4$ is a particular example of smooth toric Fano varieties
 and these manifolds play a fundamental role in the construction of examples of Calabi-Yau. Once we have a toric Fano manifold, there always  exists  a submanifold of codimension one that is a Calabi-Yau manifold
(see Section \ref{sezio toric fano}).

The toric set-up is an algebraic property that can be analyzed in terms of combinatorial algebra. Indeed,
the  classification of smooth toric Fano varieties of dimension $n$ up to isomorphism turns out to be equivalent to the
classification of combinatorial objects, namely some special polytopes in $\R^n$ up to linear unimodular transformation.

In \cite{baty2}, V. Batyrev describes a combinatorial criterion in terms of reflexive polyhedra for a hypersurface
in a toric variety to be  Calabi-Yau. He also investigates mirror symmetry in terms of an exchange of a dual pair of reflexive lattice polytopes.
Moreover, he also provides the complete biregular classification of all $4$-dimensional smooth toric Fano varieties:
there are exactly $123$ different types \cite{Baty}.

Using a computer program, M. Kreuzer and H. Sharke  are able to describe all the reflexive polyhedra
that exist in dimension four. They are about $500,000,000$  \cite{kshark}. In particular, they find more than 30,000 topological distinct
Calabi-Yau threefolds with  distinct pairs  of the Hodge numbers $(a,b)$, where $ h^{1,1}(X)=a$ and $ h^{1,2}(X)=b$
(see Section \ref{sezio calbiyau}).
Furthermore, in \cite{batykreu}  the authors find $210$ reflexive $4$-polytopes defining 68 topologically different Calabi-Yau
 varieties of dimension 3 with the Hodge number $a= 1$.

Therefore, it is  interesting to investigate the set-up of toric Fano manifolds and try to answer some questions that naturally arise.
For example, if the Hodge numbers $(a,b)$ of two Calabi-Yau manifolds  $X_1$ and $X_2$ are different,
then they are not homeomorphic. It is interesting to understand the converse.
If $X_1$ and $X_2$ have the same Hodge numbers, we wonder if they are homeomorphic or even diffeomorphic or isomorphic.

First of all, we deal with Calabi-Yau manifolds $X_1$ and $X_2$  contained in the same  toric Fano manifold.
In this specific context, we are able to prove that if $X_1$ and $X_2$ are deformation equivalent
as abstract manifolds, then they are deformation equivalent as embedded manifolds.

Then, we review the Theorems by C.T.C. Wall (Theorem \ref{theorem wall})  and by A. Kanazawa, P. M. H. Wilson  (Theorem \ref{theorem kanaza}), and we
investigate some examples of simply connected Calabi-Yau manifolds with Hodge number $a=1$.

From the point of view of moduli spaces, it is an interesting problem to understand the behaviour of the moduli space of Calabi-Yau manifolds.
In dimension two, the moduli space of K3 surfaces is an irreducible $20$ dimensional space and many properties are known.
For Calabi-Yau threefolds, it is not known whether the moduli space is irreducible or not:
M. Reid's conjecture predicts that this space should not behave too bad \cite{reid}.

Then, instead of studying the moduli space of all Calabi-Yau threefolds,  M.-C. Chang and H.I. Kim  propose to
investigate the space $M_{m,c}$ \cite{chank kim} that classifies
Calabi-Yau threefolds with fixed invariants $m$ and $c$, which are related to the invariant used by Wall (see Section \ref{sezio mirror}).
In this context, we describe an example of Calabi-Yau threefold and its mirror lying in the same $M_{5,50}$.
In particular, we provide an example of two Calabi-Yau threefolds lying in $M_{5,50}$
that are neither diffeomorphic nor deformation equivalent.

With the aim  of providing an introduction to the subject,  Section 2 is devoted to recalling some preliminaries on Calabi-Yau manifolds and toric Fano manifolds.
In Section 3, we compare the  embedded deformations of a Calabi-Yau manifold in a toric Fano 
manifold with the abstract ones. Section 4 recalls Wall's Theorem on the invariants that determine the diffeomorphism  type of closed, smooth, simply connected 6-manifolds with torsion free cohomology. We also describe some examples. In Section 5, we make some remarks on the relation between Calabi-Yau and mirror symmetry.

{\bf{Notation.}}  Throughout the paper, we will work over the field of complex numbers.

\section{Preliminaries}\label{sezio prelim}

In this section, we recall the main definitions of Calabi-Yau and toric Fano manifolds.

\subsection{Calabi-Yau Manifolds}\label{sezio calbiyau}

Let $X$ be a  complex manifold and denote by $T_X$ its holomorphic tangent bundle. $X$ is a Calabi-Yau manifold of dimension $n$ if it is a 
 projective manifold with trivial canonical bundle  
and without holomophic $p$-forms, i.e., $K_X := \Omega^n_X \cong \Oh_X$ 
and $H^0(X, \Omega^p_X)=0$ for $p$ in between $0$ and $n$.

If $X$ has dimension 3, we have   $\Omega ^3_X \cong \Oh_X$. Since $\Omega ^1_X$ is isomorphic to the dual of  $T_X $, this implies that 
$\Omega ^2_X \cong T_X$,  and, by duality,  that $H^0(X, \Omega^2_X)= H^2( X,\Oh_X) = H^1(X, \Oh_X)=H^0(X, \Omega^1_X)=0$. 

Denoting by  $h^{i,j}(X)= \operatorname{dim}_{\C} H^j(X, \Omega^i_X)$ and fixing $ h^{1,1}(X)=a$ and $ h^{1,2}(X)=b$, we can collect
the information above in  the so-called  \emph{Hodge diamond}:
 
\[
 \xymatrix@R=8pt@C=6pt{
   &   &   & 1   \\
   &   & 0 &   & 0 \\
   & 0 &   &  {a} &   & 0 \\
1  &   &  {b} &   & {b} &   & 1 \\
   & 0 &   &  {a} &   & 0 \\
   &   & 0 &   & 0 \\
   &   &   & 1.
}
\]

\bigskip

This shows that the topological type of Calabi-Yau manifold  is not uniquely determined for dimension 3. If $X_1$ and $X_2$ 
are two Calabi-Yau threefolds
with different $a$ and $b$ then they cannot be homeomorphic. 

Next, consider the case where the Hodge numbers $(a,b)$ are the same. 
Let $X_1$ and $X_2$ be two Calabi-Yau threefolds, 
with the same Hodge numbers $a$ and $b$, i.e.,  $h^{1,1}(X_i)=a$ and $h^{1,2}(X_i)=b$, for $i=1,2$. Then,
we wonder if  $X_1$ and $ X_2$  are diffeomorphic. 
Indeed, if the Calabi-Yau threefolds   are diffeomorphic, then they have the 
same numbers Hodge numbers $(a,b)$ but nothing is known about the other implication.

To understand the problem, we  focus our attention on the class of    Calabi-Yau manifolds embedded   in toric Fano manifolds.

\subsection{Toric Fano manifolds}\label{sezio toric fano}
Let $F$ be a smooth toric Fano variety of dimension $n$. 
A Fano manifold is a projective manifold $F$, whose anticanonical line bundle $-K_F := \wedge^n T_F$ is ample.

If $F$ is also  a toric variety, then  $- K_F$ is very ample (so base point free) \cite[Lemma 2.20]{oda}.
Therefore,  by Bertini's Theorem \cite[Corollary III.10.9]{hart},  the generic section of $\Oh_F(-K_F)$ gives
 a smooth connected hypersurface $X \subseteq F$, such that $X \in |-K_F|$. 
 Thus, $X$ is a smooth Calabi-Yau variety  \cite[Proposition 11.2.10]{coxlittle}. This shows that once we have a toric Fano manifold, then there always  exists  a submanifold of codimension 1 that is a Calabi-Yau manifold.

In particular, if $F$ has dimension $4$, $X$ is a smooth complex Calabi-Yau threefold.  This is actually one of the most fruitful way to construct examples of Calabi-Yau threefolds 
\cite{batykreu}.

\begin{example}
The projective space $\mathbb{P}^4$ is a smooth toric Fano manifold of dimension $4$. The general quintic hypersurface 
is a smooth   Calabi-Yau threefold. This is the most extensively studied example of Calabi-Yau threefold. In this case, it can be proved that 
$a=1$ and $b=101$.

 \end{example}

In Proposition \ref{prop. forget fano liscio}, we 
investigate  the infinitesimal deformations of  smooth complex Calabi-Yau threefolds, which are obtained as anticanonical hypersurfaces in 
a Fano manifold.

\section{Abstract vs Embedded Deformations}

In this section, we review the notion of  deformations of a submanifold $X$ in a manifold $F$. In particular, we are
interested in the infinitesimal deformations of $X$ as an  abstract manifold and in the embedded deformations of $X$ in $F$.
For more details, we refer the reader to \cite[Sections 2.4 and 3.2]{Sernesi}.

\smallskip

We denote by  $\Def_X$   the  functor of infinitesimal deformations of $X$ as an abstract variety, i.e.,
$$
\Def_X: \Art \to \Set,
$$
where $\Def_X(A)$ is the set of isomorphism classes of commutative
diagrams:
\begin{center}
$\xymatrix{X \ar[r]^i\ar[d] & X_A\ar[d]^{p_A} \\
           \Spec(\K) \ar[r] & \Spec(A),  \\ }$
\end{center}
where $i$ is a closed embedding and $p_A$ is a flat morphism.

\begin{remark}\label{remark eresma}

In our setting, $X$ is smooth, then all the fibers of $p_A$ are diffeomorphic by  Ereshman's Theorem.
 Thus,   an infinitesimal deformation of $X$ is nothing else than a deformation of the complex  structure over the same differentiable structure of $X$.
In particular,   if $X_1$ and $ X_2$ are deformation equivalent then 
they are diffeomorphic, i.e., $X_1\sim_{def} X_2 \Longrightarrow  X_1\cong_{dif}X_2$. 
The converse is not true: $X_1\cong_{dif}X_2 \not \Longrightarrow X_1\sim_{def} X_2 $. 
There are examples of diffeomorphic Calabi-Yau threefolds that are not deformation equivalent  \cite{gross, ruan}.

\end{remark}

\begin{remark}

If $X$ is a Calabi-Yau manifold, then Bogomolov-Tian-Todorov Theorem implies that the functor $\Def_X$ is smooth. This property implies that the moduli space is smooth at the point corresponding to $X$.
 
\end{remark}

We denote by $H^F_X$ the functor of infinitesimal embedded deformations of $X$ in $F$, i.e.,
 $$
H^F_X: \Art \to \Set,
$$
where $H^F_X(A)$ is the set of commutative
diagram:
\begin{center}
$\xymatrix{X \ar[r]^i\ar[d] & X_A\ar[d]^{p_A} \subset F \times \Spec(A) \\
           \Spec(\K) \ar[r] & \Spec(A), &  \\ }$
\end{center}
where $i$ is a closed embedding, $X_A \subset F \times \Spec(A)$ and $p_A$ is a flat morphism induced by the projection
$F \times \Spec(A) \to \Spec(A)$.

In particular, the following forgetful morphism of functors is well defined:
\[
 \phi: H^F_X \to \Def_X;
\]
moreover, the image of an infinitesimal deformation of $X$ in $F$ is  the isomorphism class of the deformation
of $X$, viewed as an abstract deformation  \cite[Section 3.2.3]{Sernesi}.

 \begin{example}

Let $n \geq 4$ and   $X$  be the general anticanonical
hypersurface in $\mathbb{P}^n$. Note that  $\mathbb{P}^n$ is a  smooth toric Fano variety and   $X$  a smooth  Calabi-Yau manifold.

For every Calabi-Yau manifold $X$ in a projective space $\mathbb{P}^n$, the embedded deformations of $X$ in $\mathbb{P}^n$ are unobstructed
\cite[Corollary A.2]{huy}.

Therefore, the functor $ \Def_X$ and the morphism $\phi$ are both smooth and this implies that $H^{\mathbb{P}^n}_X$ is also smooth
\cite[Corollary 2.3.7]{Sernesi}.

In particular, this implies that all the infinitesimal deformations of  the general anticanonical
hypersurface $X$   as an abstract variety
are obtained as embedded deformations of $X$ inside $\mathbb{P}^n$.  The following proposition shows that a similar property   is true
for any smooth toric Fano variety and not only for $\mathbb{P}^n$.

\end{example}

\begin{proposition}\label{prop. forget fano liscio}
Let $F$ be a smooth toric Fano variety with $\operatorname{dim}  F >3$ and denote by $X$
a smooth connected hypersurface in $F$ such that $X \in |-K_F|$. Then,
the forgetful morphism
\[
 \phi: H^F_X \to \Def_X
\]
is smooth.

\end{proposition}

\begin{proof}
The varieties $F$ and $X$ are both smooth, so we have the exact sequence
\[
 0 \to T_X \to T_{F|X} \to N_{X/F} \to 0
\]
that induces the following exact sequence in cohomology, namely:
\[
\cdots \! \to H^0(X, N_{X/F}) \stackrel{\alpha}{\to}\!
 H^1(X, T_X) \to\! H^1(X, T_{F|X}) \to H^1(X, N_{X/F}) \stackrel{\beta}{\to}\!
 H^2(X, T_X) \to \! \cdots .
\]
The morphism $\alpha$ is the map induced by $\phi$ on the tangent spaces and $\beta$ is an
obstruction map for $\phi$ \cite[Proposition 3.2.9]{Sernesi}.
Applying the standard smoothness criterion \cite[Proposition 2.3.6]{Sernesi}, it
is enough to prove that $\alpha$ is surjective  and $\beta$ is injective; in particular, it  suffices to prove
that  $ H^1(X, T_{F|X})=0$.

For this purpose, consider the exact sequence
\[
 0 \to \Oh_F(-X) \to \Oh_F \to \Oh_X \to 0
\]
and tensor it with $T_F$, thus yielding
\[
 0 \to T_F \otimes \Oh_F(-X) \to T_F \to T_{F|X} \to 0
\]
and the induced exact sequence in cohomology, namely:
\[
\cdots \to
 H^1(F, T_F) \to H^1(F, T_{F|X}) \to H^2(F, T_F \otimes \Oh_F(-X)) \to
 H^2(F, T_F)  \to \cdots .
\]
If $F$ is a smooth toric Fano variety, then $H^i(F,T_F)=0$, for all $i>0$ \cite[Proposition 4.2]{brion}.
Since $\Oh_F(-X) \cong \Oh_F(K_F)$, we are reduced to prove that $H^2(F, T_F \otimes \Oh_F(K_F))= 0$.
This follows from Lemma \ref{lemma H2 0}.

\end{proof}

\begin{lemma}\label{lemma H2 0}
 Let $F$ be a smooth toric Fano variety with $\operatorname{dim}  F >3$. Then the following holds:
$$
H^2(F, T_F \otimes \Oh_F(K_F))=0.
$$
\end{lemma}

\begin{proof}

As for projective space, there exists a generalized Euler exact sequence for the tangent bundle of toric varieties
\cite[Theorem 8.1.6]{coxlittle}:
\[
 0 \to Pic(F) \otimes_{\Z} \Oh_F    \to  \oplus_i \Oh_F (D_i) \to T_F \to 0,
\]
where $K_F= - \sum_i D_i$ \cite[Theorem 8.2.3]{coxlittle}.  We note also that  $ Pic(F) \otimes_{\Z} \Oh_F \cong \Oh_F^{rank}$, 
where $rank$ denotes the $rank$ of $Pic(F)$.
By tensoring with $ \Oh_F(K_F)$, we obtain
 \[
 0 \to Pic(F) \otimes_{\Z} \Oh_F(K_F)    \to  \oplus_i \Oh_F (D_i+K_F) \to T_F \otimes \Oh_F(K_F) \to 0
\]
 and so
\[
\cdots \to
 H^2(F, Pic(F) \otimes_{\Z} \Oh_F(K_F)) \to
\]
\[
\cdots \to \oplus_i H^2(F, \Oh_F (D_i+K_F)) \to H^2(F,  T_F \otimes \Oh_F(K_F))\!
 \to H^3(F, Pic(F) \otimes_{\Z} \Oh_F(K_F)) \to \cdots
\]

Since $-K_F$ is ample, by Kodaira
vanishing Theorem, $H^j(F, \Oh_F)=0, j>0$.  Moreover, by Serre duality $H^j(F, \Oh_F(K_F))=0, j\neq \operatorname{dim}F$.
Therefore, if $\operatorname{dim}   F >3$, then $ H^2(F, Pic(F) \otimes_{\Z} \Oh_F(K_F))= H^3(F, Pic(F) \otimes_{\Z} \Oh_F(K_F))=0$ and
\[
 \oplus_i H^2(F, \Oh_F (D_i+K_F)) \cong H^2(F,  T_F \otimes \Oh_F(K_F)).
\]
By Serre duality,  $ H^2(F, \Oh_F (D_i+K_F))\cong H^{\operatorname{dim} F-2}(F, \Oh_F (-D_i))^\nu$, for all $i$.

Using the following exact sequence
\[
 0 \to \Oh_F (-D_i)    \to    \Oh_F  \to   \Oh_{D_i}   \to 0 
\]
and the fact that  $H^j(F, \Oh_F)=0,$ for $j>0$, we conclude that  $H^{\operatorname{dim} F-2}(F, \Oh_F (-D_i)) \cong H^{\operatorname{dim} F-3}(D_i, \Oh_{D_i} )$, for all $i$.

Therefore, we are left to prove that $\oplus_i H^{\operatorname{dim} F-3}(D_i, \Oh_{D_i} ) =0$.

Consider the following exact sequence on a toric variety \cite[Theorem 8.1.4]{coxlittle}
\[
 0 \to \Omega^1_F    \to  M \otimes_Z \Oh_F  \to \oplus_i \Oh_{D_i}   \to 0,
\]
where $M$ is a lattice related to the toric structure; here we only need that $M \otimes_Z \Oh_F \cong \Oh_F ^r$, for some $r \in\N$.

This induces
\[
\cdots  \to  H^{\operatorname{dim} F-3}(F,M \otimes_Z \Oh_F)  \to \oplus_i H^{\operatorname{dim} F-3}(D_i, \Oh_{D_i} ) 
\to   H^{\operatorname{dim}F-2}(F, \Omega^1_F )   \to \cdots.
\]
Since $H^j(F, \Oh_F)=0,$ for $j>0$ and $\operatorname{dim} F>3$, we have
$H^{\operatorname{dim} F-3}(F,M \otimes_Z \Oh_F) = H^{\operatorname{dim} F-2}(F, \Omega^1_F ) =0$
\cite[Theorem 9.3.2]{coxlittle}. This implies $\oplus_i H^{\operatorname{dim} F-3}(D_i, \Oh_{D_i} ) =0$.

\end{proof}

\begin{remark} \label{remark defo equivalent fano}
Proposition \ref{prop. forget fano liscio} shows that  all the infinitesimal deformations of $X$ as an abstract variety are obtained as
infinitesimal deformations of $X$ inside the smooth toric Fano manifold $F$.
Moreover, since every deformation of a Calabi-Yau manifold is smooth (Bogomolov-Tian-Todorov Theorem),
we conclude that  the deformations of $X$ inside $F$ are also smooth.

\end{remark}

\section{Diffeomorphic Three-dimensional Calabi-Yau varieties}

In this section, we focus on the diffeomorphism class of three dimensional Calabi-Yau manifolds.

In 1966,   C.T.C. Wall described the invariants that determine the diffeomorphism
type of closed, smooth, simply connected 6-manifolds with torsion free  cohomology.

\begin{theorem}\label{theorem wall} \cite{wall}
Diffeomorphism classes of simply-connected, spin, oriented, closed 6-manifolds $X$
with torsion-free cohomology
correspond bijectively to isomorphism classes of systems of invariants consisting of

\begin{enumerate}
\item  free Abelian groups $H^2 (X,\Z)$ and $H^3(X,\Z)$,
\item  a symmetric trilinear form $\mu \! :\! H^2 (X,\Z)^{\otimes 3} \to H^6 (X,\Z) \cong \Z$ defined by
$\mu(x,y,z)\! := x \cup y \cup z$,
\item a linear map $p_1: H^2 (X,\Z)  \to H^6 (X,\Z) \cong \Z$, defined by
$p_1(x):= p_1(X)  \cup x$, where $p_1(X) \in H^4 (X,\Z)$ is the first Pontrjagin class of $X$, satisfying,
\end{enumerate}
 for any $x,y \in H ^2 (X,\Z)$, the following conditions

\[
 \mu(x,x,y)+ \mu(x,y,y) \equiv 0 \pmod{2} \qquad \qquad  4\mu(x,x,x) -p_1(x) \equiv 0 \pmod{24}.
\]

 \end{theorem}
The symbol  $\cup$ denotes the cup product of differential forms and
the isomorphism $H^6 (X,\Z) \cong \Z$ above is given by pairing a cohomology class with the fundamental class of $X $ with natural orientation.

\medskip

Let $X$ be a Calabi-Yau threefold. In  \cite{Kanazawa}, the authors investigate the interplay between the trilinear form $\mu$ and the
Chern classes $c_2 (X)$ and $c_3 (X)$ of $X$, providing the following numerical relation.

\begin{theorem}\label{theorem kanaza}\cite{Kanazawa}
Let $(X,H)$ be a very ample polarized Calabi-Yau threefold, i.e., $x=H$ is a very ample divisor on $X$. Then the following inequalities holds:
\begin{equation}\label{equa kanaza}
 -36 \mu(x,x,x) -80 \leq \frac{c_3(X)}{2}= h^{1,1}(X) - h^{2,2} \leq 6 \mu(x,x,x) +40.
\end{equation}
\end{theorem}

Note, that if  $X$ is a Calabi-Yau threefold, then $p_1(X)= -2c_2 (X) \in H^4 (X,\Z)$ and
\[
 \int_X c_3 (X)= \chi(X) = \sum_{i=0}^6 \operatorname{dim} H^i(X, \R)= 2h^{1,1}(X) - 2h^{1,2}.
\]

\begin{remark}
 By Wall's Theorem, if $X$ is simply-connected, spin, oriented, closed 6-manifolds
with torsion-free cohomology, then the diffeomorphism class is determined by  the free Abelian groups $H^2 (X,\Z)$ and $H^3(X,\Z)$,
and the form $\mu$ and $p_1$. For any data we have a diffeormphism class. If $X$ is Calabi Yau, then  $\mu$ and $p_1$ have to satisfy the
numerical conditions of Equation (\ref{equa kanaza}). Note that, having $\mu$ and $p_1$ on $X$ that satisfy all the numerical conditions,
it does not imply that $X$ is a Calabi Yau.

\end{remark}

In particular, let  $X_1$ and $X_2$ be two simply connected  Calabi-Yau threefolds with torsion-free cohomology
and the same Hodge numbers  $h^{1,1}(X)=a$ and $h^{1,2}(X)=b$. To be diffeomorphic, they should have the same $\mu$ and $p_1$,
that satisfy the numerical conditions.

\begin{corollary}
Let $X$ be a Calabi-Yau threefold, with torsion-free cohomology
and  $h^{1,1}(X)=1$ and $h^{1,2}(X)=h^{2,1}(X)=b$, for some $b \in \N$; hence, we have $H ^2 (X,\Z) \cong \Z$ and $H ^3 (X,\Z) \cong \Z^{2+b}$. Fix a generator $H \in H ^2 (X,\Z)$ and set $\mu(H,H,H)= m \in \Z$. Then the following holds:
\[
 m \geq \frac{b-81}{36}.
 \]
\end{corollary}

\proof Set $p_1(X)= -2c_2 (X) \in H^4 (X,\Z)$; so there exists $c \in \Z$ such that $c_2 (X)=c H^*$. Therefore, the linear form $p_1$ reduces to
\[p_1: H^2 (X,\Z)  \to H^6 (X,\Z) \cong \Z \qquad
p_1(xH):= -2 c_2(X) \cup xH= -2cx H^* \cup H= -2cx.
\]
The numerical  constraints of Theorem \ref{theorem wall} reduce to
\[
mx^2 y+ mxy^2 \equiv 0 \pmod{2} \qquad \qquad    2mx^3 +cx \equiv 0 \pmod{12}
\]
for any $x,y \in \Z$.
The former condition is always verified while the latter congruence is equivalent to  $2m +c\equiv 0 \pmod{12}$.

As for the numerical restriction of Theorem \ref{theorem kanaza}, Equation \eqref{equa kanaza} reduces to
\[
 -36 \mu(x,x,x) -80 \leq 1-b \leq 6 \mu(x,x,x) +40.
\]
and so
\[
 -36 m -80 \leq 1-b \leq 6 m +40.
\]
In particular,
\[
b \leq  81 +36 m   \qquad -39 -6m \leq b;
 \]
\[
 m \geq \frac{b-81}{36}  \qquad m \geq \frac{-39-b}{6}.
 \]
Since $b$ is positive, they reduce to
\[
 m \geq \frac{b-81}{36}.
 \]
\endproof

\begin{example}
If $m=5$, then $b \leq 261$. In \cite{kap}, Appendix 1, there are three examples that satisfy this condition, 
namely $b=51, 101, 156$. For $b=101$ we obtain the general quintic threefold in $\mathbb{P}^4$. Projective models for 
the remaining two are still mysterious, as indicated by the question mark in the table in \cite{kap}.
\end{example}

 \section{Some remarks on the moduli space of Calabi-Yau manifolds and mirror symmetry}\label{sezio mirror}

Let $X$ be a Calabi-Yau variety. Denote by $H$ a primitive ample divisor. As in \cite{chank kim}, let $M_{m,c}$
be the space of polarized varieties $(X,H)$ such that $H^3=m$ and $c_2(X)H=c$ for integers $m$ and $c$. Little
is known on the geometric structure of $M_{m,c}$. Some information can be found in \cite{chank kim}.

Here we make the following remarks. Let $X$ be a general quintic in ${\mathbb P}^4$. A hyperplane section on $X$
is a (very) ample divisor $H$ such that $H^3=5$. On the Calabi-Yau manifold $X$ the Grothendieck-Riemann-Roch
Theorem reads as follows:
$$
\chi(H)=\frac{H^3}{6} + \frac{1}{12}c_2(X)H.
$$

Since $H$ is a divisor on $X$, we have
$$
\chi({\mathcal O}_X) + \chi({\mathcal O_H}(H))=\chi(H).
$$

The first term on the left-hand side is zero because $X$ is a Calabi-Yau; the second term can be computed via Noether's formula, namely:
$$
\chi({\mathcal O_H}(H))=\frac{K_H^2+c_2(H)}{12}.
$$

A linear section of a quintic is a quintic surface in three-dimensional projective space. As well known,
the second Betti number is 53, so the Euler characteristic is 55. Therefore, we get
$$
\chi({\mathcal O_H}(H))=5.
$$

Hence, we get
$$
5=\frac{H^3}{6} + \frac{1}{12}c_2(X)H,
$$
which yields $c_2(X)H=50$.

This means that the pair $(X,H)$ belongs to the space $M_{5,50}$, where $X$ is a quintic in ${\mathbb P}^4$ and $H$ is a
hyperplane section.

The Hodge numbers of $X$ are given by $(a,b)=(1,101)$. The Hodge numbers of a mirror manifold $X'$ are given by $(101,1)$.

\begin{proposition} There exists a primitive ample divisor $D$ on $X'$ such that $(X',D)$ belongs to $M_{5,50}$.
\end{proposition}

\begin{proof}

In fact, as recalled in \cite{BvGT}, a mirror of $X$ can be found as a crepant resolution of a singular
quintic in ${\mathbb P}^4$. Denote by $D$ the pull-back of the hyperplane divisor on projective space.
Clearly, $D^3=5$. Now, we need to compute
$$
 \chi(D)=\frac{D^3}{6} + \frac{1}{12}c_2(X')D.
$$

Like before, we have
$$
\chi({\mathcal O_D}(D))=\frac{K_D^2+c_2(D)}{12}.
$$

Notice that
$$
K_D^2=D^2D=5.
$$

As mentioned before, the divisor $D$ is the pull-back of the hyperplane divisor.
We can take a member of it that does not intersect the blown up locus. Thus, $c_2(D)$ is again
the Euler characteristic of a quintic surface in three-dimensional projective space. Therefore, the claim follows.

\end{proof}

\begin{example}
For the general quintic threefold $X$ in $\mathbb{P}^4$, we have $a=1$, $b=101$,  $m=5$ and $c=50$, that
satisfy the previous conditions. Therefore,
$X$ lies in the space $M_{5,50}=\{(X, H)\mid \ H^3=5, c_2 (X) \cdot H = 50\}$ introduced in \cite{chank kim}.  The mirror $\tilde{X}$ of $X$
is a smooth Calabi-Yau
threefolds with $a=101$, $b=1$, $m=5$ and $c=50$. So it lies in the same space  $M_{5,50}$ but the Hodge numbers are exchanged: this implies
that $X$ and $\tilde{X}$ are neither diffeomorphic nor deformation equivalent!
\end{example}

\end{document}